\documentclass[12pt]{article}
\usepackage{amssymb}
\usepackage{amsmath}
\usepackage{amsthm}
\usepackage[all]{xy}
\usepackage[francais]{babel}

\newtheorem{theo}{Th\'eor\`eme}[section]
\newtheorem{prop}[theo]{Proposition}
\newtheorem{lem}[theo]{Lemme}
\newtheorem{cor}[theo]{Corollaire}

\newtheorem*{theo*}{Th\'eor\`eme}

\renewcommand{\a}{\alpha}
    
\renewcommand{\P}{\mathbb{P}} 
\newcommand{\N}{\mathbb{N}}  
\newcommand{\kk}{\mathbf{k}}
\renewcommand{\Im}{\,\mathrm{Im}\,}

\renewcommand{\O}{\mathcal{O}}                       
\newcommand{\I}{\mathcal{I}} 
\newcommand{\E}{\mathcal{E}}
\newcommand{\F}{\mathcal{F}}
\newcommand{\G}{\mathcal{G}}
\renewcommand{\H}{\mathcal{H}}
\newcommand{\U}{\mathcal{U}}

\DeclareMathOperator{\RatCurves}{RatCurves}
\DeclareMathOperator{\Univ}{Univ} 
\DeclareMathOperator{\birat}{bir}
\DeclareMathOperator{\Pic}{Pic}   
\DeclareMathOperator{\Hom}{Hom}
\DeclareMathOperator{\codim}{codim}
\DeclareMathOperator{\rg}{rg}      
\DeclareMathOperator{\Sym}{Sym} 

\title{\textbf{ Caract\'erisations des espaces projectifs et des quadriques}}
\date{}
\author{Matthieu Paris\\
\small{Institut Fourier} \\
\small{Universit\'e Joseph Fourier Grenoble 1}\\
\small{38 402 Saint Martin d'H\`eres, France }\\
\small{e-mail: matthieu.paris@ujf-grenoble.fr}}

\begin{document}
\maketitle

\section{Introduction}

L'objectif de ce texte est de caract\'eriser les espaces projectifs et les quadriques par des propri\'et\'es de positivit\'e de leur fibr\'e tangent. De nombreux r\'esultats existent d\'ej\`a en ce sens, et le lecteur pourra trouver dans \cite{ADK} une pr\'esentation des principaux d'entre eux. Dans cet article, Araujo, Druel et Kov\'acs d\'emontrent la caract\'erisation  suivante, issue d'une conjecture de Beauville (\cite{Bea}).

\begin{theo*}\cite[Theorem 1.1, Theorem 6.3]{ADK}
Soient $X$ une vari\'et\'e projective lisse complexe de dimension $n\geq 1$, et $L$ un fibr\'e en droites ample sur $X$. Soit $p\geq 1$ un entier. On suppose que l'une des deux hypoth\`eses suivantes est satisfaite :
\begin{enumerate}
	\item $ H^0(X, \wedge ^p T_X \otimes L^{-p}) \neq 0$
	\item $H^0(X, T_X ^{\otimes p} \otimes L^{-p}) \neq 0 \ \text{ et }\ \rho(X)=1$
\end{enumerate}

Alors ou bien  $( X , L) \cong (\P^1, \O_{\P^1}(2))$, ou bien $( X , L) \cong (\P^n, \O_{\P^n}(1))$, ou bien $( X , L) \cong (Q_n, \O_{Q_n}(1))$, o\`u $Q_n\subset \P^{n+1}$ est une quadrique lisse de dimension $n$. 
\end{theo*}  

On se propose ici de g\'en\'eraliser ce r\'esultat en affaiblissant ses hypoth\`eses de la fa\c{c}on suivante. 
\begin{theo}\label{theo:1}
Soient $X$ une vari\'et\'e projective lisse complexe de dimension $n\geq 1$, et $L$ un fibr\'e en droites ample sur $X$. Soit $p\geq 1$ un entier. On suppose que 
\[
H^0(X, T_X ^{\otimes p} \otimes L^{-p}) \neq 0
\]
Alors la conclusion du th\'eor\`eme ci-dessus est encore vraie. 
\end{theo}

Pour d\'emontrer ce th\'eor\`eme, on commence par remarquer que $X$ est unir\'egl\'ee d'apr\`es le th\'eor\`eme de Miyaoka (\cite[Corollary 8.6]{Miy}). On choisit alors sur $X$ une famille couvrante minimale de courbes rationnelles $H$, dont on montre dans un premier temps qu'elle est compl\`ete. On proc\`ede ensuite \`a l'\'etude du quotient $H$-rationnellement connexe (ou quotient $H$-rc) $\pi_0 :X_0 \rightarrow Y_0$ de $X$ (on renvoie encore une fois le lecteur \`a \cite{ADK} pour les d\'efinitions et les propri\'et\'es de ces objets). Cette \'etude permet de se ramener \`a deux cas particuliers: le cas d'un fibr\'e projectif sur une courbe, et le cas d'un fibr\'e en quadriques sur une courbe. Ces deux cas font l'objet de th\'eor\`emes d'annulation (th\'eor\`emes \ref{theo:3} et \ref{theo:4}) qui donnent la conclusion. 
 
On montrera \'egalement par les m\^emes m\'ethodes la caract\'erisation des espaces projectifs suivante, qui constitue un outil important dans la d\'emonstration du th\'eor\`eme \ref{theo:1}.

\begin{theo}\label{theo:2}
Soient $X$ une vari\'et\'e projective lisse complexe de dimension $n\geq 1$, et $L$ un fibr\'e en droites ample sur $X$. Soient $p\geq 1$ et $k>p$ deux entiers.
On suppose que 
\[
H^0(X, T_X ^{\otimes p} \otimes L^{-k}) \neq 0
\]
Alors $X\cong \P^n$, et $L\cong \O_{\P^n}(1)$. 
\end{theo}

\textit{ Organisation du texte. } La section $2$ rassemble quelques r\'esultats pr\'eliminaires, puis on montre 
dans la section $3$ deux th\'eor\`emes d'annulation sur les fibr\'es projectifs et fibr\'es en quadriques. La section 4 est consacr\'ee \`a la preuve des deux r\'esultats principaux \ref{theo:1} et \ref{theo:2}.

\textit{ Conventions et notations. }  Toutes les vari\'et\'es consid\'er\'ees ici seront d\'efinies sur le corps des nombres complexes. Si $X$ est une vari\'et\'e et si $x\in X$ on notera $\kk(x)=\O_{X,x}/ \mathfrak{m}_{X,x}$ le corps r\'esiduel. 
Si $E$ est un fibr\'e vectoriel sur $X$, $\P(E)$ d\'esignera le fibr\'e projectif des hyperplans dans les fibres de $E$, autrement dit $\P(E)=\textbf{Proj}( \Sym(E))$. 

\section{R\'esultats pr\'eliminaires}
Le premier lemme de cette section d\'ecrit les puissances tensorielles d'un fibr\'e vectoriel apparaissant au milieu d'une suite exacte courte.
  
\begin{lem}\label{lem:1}
Soit $X$ une vari\'et\'e, et $K,F,G$ des fibr\'es vectoriels sur $X$. On suppose qu'on a une suite exacte
\[
0 \rightarrow K \rightarrow F \stackrel{\pi}{\rightarrow} G  \rightarrow 0.
\]  
Alors il existe pour tout entier $p\geq 1$ une filtration de $F^{\otimes p}$ par des sous-fibr\'es vectoriels
\[
F^{\otimes p}=F_p^0 \supset F_p^1 \supset \cdots \supset F_p^m \supset F_p^{m+1}=0
\]
ayant la propri\'et\'e suivante : 
\[ \forall l \in \{0,\ldots,m\} \ , \ \exists i\in  \{0,\ldots,p\} \ , \ F_p^l/F_p^{l+1} \cong  K^{\otimes i} \otimes  G^{\otimes p-i}. \]
\end{lem}
\begin{proof}
On proc\`ede par r\'ecurrence sur $p$. Si $p=1$ on obtient la filtration voulue en posant $F_1^0:=F$,  $F_1^1:=K$, et $F_1^2 :=0$.  

Si $p\geq 2$ on suppose qu'il existe une filtration $F^{\otimes p-1}$ par des sous-fibr\'es vectoriels
\[
F^{\otimes {p-1}}=F_{p-1}^0 \supset F_{p-1}^1 \supset \cdots \supset F_{p-1}^m \supset F_{p-1}^{m+1}=0
\] 
ayant la propri\'et\'e demand\'ee. On a une suite exacte de fibr\'es vectoriels
\[
0 \rightarrow K \otimes F^{\otimes p-1} \rightarrow F^{\otimes p} \stackrel{\varphi}{\rightarrow} G \otimes F^{\otimes p-1} \rightarrow 0.
\]
On pose pour tout $l\in \{0, \ldots m+1\}$, 
\[
F_p^l:= \varphi^{-1}(G \otimes F_{p-1}^l),
\] 
et pour tout  $l\in \{m+1, \ldots 2m+2\}$,
\[
F_p^l:= K \otimes F_{p-1}^{l-m-1}
\]
On notera que pour $l=m+1$ les deux d\'efinitions coïncident, on a $F_p^{m+1} =  K \otimes F^{\otimes p-1}$.

On a alors si $l\in \{0, \ldots m\}$
\[
F_p^l/F_p^{l+1} \cong \varphi^{-1}(G \otimes F_{p-1}^l) / \varphi^{-1}(G \otimes F_{p-1}^{l+1}) \cong G \otimes (F_{p-1}^l/ F_{p-1}^{l+1}) \cong G \otimes K^{\otimes i} \otimes  G^{\otimes p-i-1},
\]
pour un entier $i\in\{0, \ldots , p-1\}$. De m\^eme si $l\in \{m+1, \ldots 2m+1\}$ on a 
\[
F_p^l/F_p^{l+1} \cong(K \otimes F_{p-1}^{l-m-1}) / (K \otimes F_{p-1}^{l-m}) \cong K \otimes (F_{p-1}^{l-m-1}/  F_{p-1}^{l-m})\cong K \otimes K^{\otimes j} \otimes  G^{\otimes p-j-1},
\]
pour un entier $j\in\{0, \ldots , p-1\}$. 
On a donc une filtration 
\[
F^{\otimes p}=F_p^0 \supset F_p^1 \supset \cdots \supset F_p^{2m+1} \supset F_p^{2m+2}=0
\]
avec des quotients de la forme voulue.
\end{proof}

On en d\'eduit le r\'esultat suivant.

\begin{cor}\label{cor:1}
Soient $\pi: X \rightarrow Y$ un morphisme lisse entre deux vari\'et\'es lisses, et $L$ un fibr\'e en droites sur $X$.  Soit $p\geq 1$ un entier tel que 
\[
H^0(X, T_X ^{\otimes p} \otimes L) \neq 0.
\]
Alors il existe un entier $i \in \{0,\ldots,p\}$  tel que 
\[
H^0(X, T_{X/Y}^{\otimes i} \otimes  (\pi^\ast T_Y)^{\otimes p-i}\otimes L) \neq 0.
\]
\end{cor}

\begin{proof} 
On a la suite exacte suivante
\[
0 \rightarrow  T_{X/Y} \rightarrow T_X \rightarrow \pi^\ast T_Y \rightarrow 0.
\]
Par cons\'equent d'apr\`es le lemme \ref{lem:1} il existe une filtration de $T_X^{\otimes p}\otimes L$ 
\[
 T_X^{\otimes p} \otimes L = F^0 \supset F^1 \supset \cdots \supset F^m \supset F^{m+1}=0
 \]
v\'erifiant que pour tout $l \in \{0,\ldots,m\}$ il existe $i\in\{0, \cdots p\}$ tel que $F^l/F^{l+1} \cong  T_{X/Y}^{\otimes i} \otimes  (\pi^\ast T_Y)^{\otimes p-i} \otimes L$. Or on a $H^0(X,F^0) \neq 0$, donc il existe un entier $l \in \{0,\ldots,m\}$ tel que $H^0(X,F^l/F^{l+1}) \neq 0$, ce qui donne le r\'esultat. 
\end{proof}

Il sera int\'eressant par la suite de conna\^itre les conditions sur des entiers $n,p,k$ pour que le fibr\'e vectoriel $T_{\P^n}^{\otimes p}(-k)$ ait des sections globales non nulles. Les voici :

\begin{lem}\label{lem:2}
Soient $n,p,k$ trois entiers, avec $n\geq 1$ et $p\geq 1$. Alors
\[
H^0(\P^n,  T_{\P^n}^{\otimes p}(-k)) \neq 0 \quad \Longleftrightarrow \quad k \leq \frac{p(n+1)}{n}
\]
\end{lem}

\begin{proof} 
Supposons que $k \leq \frac{p(n+1)}{n}$ et consid\'erons la division euclidienne de $p$ par $n$ : $p=nq +r$, avec $q\geq 0$ et $0\leq r < n$. On a 
\[
H^0(\P^n,  T_{\P^n}^{\otimes p}(-k)) = H^0(\P^n,  [T_{\P^n}^{\otimes n}(-n-1)]^{\otimes q} \otimes  T_{\P^n}^{\otimes r}(-l))    \ \text{ avec } l=k-q(n+1).
\] 
Il existe un morphisme injectif 
\[
 \O_{\P^n}(n+1) \cong \det(T_{\P^n}) \hookrightarrow T_{\P^n}^{\otimes n} ,
\]
ce qui permet d'affirmer que  $H^0(\P^n,  [T_{\P^n}^{\otimes n}(-n-1)]^{\otimes q})\neq 0$.

De plus $l =k-q(n+1) \leq \frac{p(n+1)}{n} -q(n+1) = r + \frac{r}{n}$, avec $r<n$, donc $l\leq r$. Comme le fibr\'e $T_{\P^n}(-1)$ a des sections non nulles on a donc  
\[
H^0(\P^n,  T_{\P^n}^{\otimes r}(-l)) = H^0(\P^n,  [T_{\P^n}(-1)]^{\otimes r}(r-l)) \neq 0 ,
\] 
et on obtient avec ce qui pr\'ec\`ede que $H^0(\P^n,  T_{\P^n}^{\otimes p}(-k)) \neq 0$.

R\'eciproquement, supposons que $H^0(\P^n,  T_{\P^n}^{\otimes p}(-k)) \neq 0$. Le fibr\'e  $T_{\P^n}^{\otimes p}(-k)$ est semi-stable d'apr\`es \cite[Lemma 1.4.5]{Huy} et \cite[Theorem 3.1.4]{Huy} . Pour qu'il admette une section non nulle, il est donc n\'ecessaire qu'il ait une pente positive ou nulle. On rappelle que la pente d'un faisceau sans torsion $\F$, relativement \`a un fibr\'e en droites ample $L$, est d\'efinie par
\[
\mu(\F):= \frac{c_1(\F)\cdot c_1(L)^{n-1}}{\rg(\F)}.
\]
On a donc finalement
\[
\mu(T_{\P^n}^{\otimes p}(-k)) = p\left(\frac{n+1}{n}\right) -k \geq 0.
\]
\end{proof}

On termine cette section par un lemme technique sur les fibr\'es projectifs sur $\P^1$. Il sera utilis\'e \`a plusieurs reprises pour d\'emontrer les th\'eor\`emes \ref{theo:1} et \ref{theo:2}.

\begin{lem}\label{lem:3}
Soient $E$ un fibr\'e vectoriel ample de rang $r+1 \geq2$ sur $\P^1$, et $N$ un fibr\'e en droites nef sur $\P^1$. On note $\pi: X=\P(E) \rightarrow \P^1$ le fibr\'e projectif associ\'e \`a $E$. Soit $T:=\O_{\P^1}(\a_0) \oplus \cdots \oplus \O_{\P^1}(\a_m)$ avec des entiers $ 2 \geq \a_0 \geq \ldots \geq \a_m  $. \'Etant donn\'es deux entiers $p,q \in \N$, avec $p\geq 1$, on suppose qu'il existe un entier $i\in \{0,\ldots,p\}$ tel que 
\[
H^0(X, T_{X/\P^1} ^{\otimes i}\otimes \pi^\ast T^{\otimes p-i} \otimes \O_{\P(E)}(-p-q)\otimes \pi^\ast N^{-1}) \neq 0
\]
Alors $E = \O_{\P^1}(1) \oplus  \O_{\P^1}(1)$, $ p=2i$, $q=0$, et $N =\O_{\P^1}$. En particulier $X \cong \P^1 \times \P^1$.
\end{lem}

\begin{proof}
Notons d'abord que l'hypoth\`ese entra\^ine en se restreignant \`a une fibre g\'en\'erale $f$ de $\pi$ que
\[
H^0(f , T_{f}^{\otimes i} \otimes \O_{\P(E)}(-p-q)|_f) \neq 0. 
\]
Comme $f\cong\P^r$ et $\O_{\P(E)}(-p-q)|_f\cong \O_{\P^r}(-p-q)$, on en d\'eduit que $i\geq \frac{r}{r+1} (p+q)$ d'apr\`es le lemme \ref{lem:2}. En particulier on notera que $2i\geq p$.

Comme $E$ et $N$ sont respectivement ample et nef on peut \'ecrire $E = \O_{\P^1}(a_0)\oplus \cdots \oplus \O_{\P^1}(a_r)$ avec $1\leq a_0 \leq \cdots \leq a_r$, et $N = \O_{\P^1}(n)$ avec $n\geq 0$. 
Le morphisme surjectif de faisceaux $E \rightarrow \O_{\P^1}(a_r)$ donne lieu \`a une section $\sigma$ de la projection naturelle $\pi : X \rightarrow \P^1$ telle que $\sigma^\ast \O_{\P(E)}(1)\cong \O_{\P^1}(a_r)$. Notons $\ell$ l'image de $\P^1$ par cette section.
On a la suite exacte
\[
0 \rightarrow \O_X(1)\otimes \I_\ell \rightarrow \O_X(1) \rightarrow \O_X(1)|_\ell \rightarrow 0.
\]
En appliquant $\pi_\ast$ \`a cette suite exacte on obtient 
\[
0 \rightarrow \pi_\ast[\O_X(1)\otimes \I_\ell] \rightarrow E = \O_{\P^1}(a_0)\oplus \cdots \oplus \O_{\P^1}(a_r) \rightarrow \O_{\P^1}(a_r) \rightarrow 0,
\]
donc $\pi_\ast[\O_X(1)\otimes \I_\ell] \cong \O_{\P^1}(a_0)\oplus \cdots \oplus \O_{\P^1}(a_{r-1})$. 
Par ailleurs l'application naturelle 
\[
\pi^\ast \pi_\ast[\O_X(1)\otimes \I_\ell] \rightarrow \O_X(1)\otimes \I_\ell
\]
est surjective, et induit en restriction \`a la courbe $\ell$ un isomorphisme
\[
(\pi^\ast \pi_\ast[\O_X(1)\otimes \I_\ell])|_\ell \cong \O_{\P^1}(a_0)\oplus \cdots \oplus \O_{\P^1}(a_{r-1}) \stackrel{\sim}{\rightarrow} (\O_X(1) \otimes \I_\ell)|_\ell.
\]
On en d\'eduit que
\[
N_{\ell/X}\cong (\I_\ell / \I^2_\ell)^\ast \cong (\I_\ell|_\ell)^\ast \cong \O_X(1)|_\ell \otimes [\O_{\P^1}(a_0)\oplus \cdots \oplus \O_{\P^1}(a_{r-1})]^\ast ,
\]
d'o\`u  $N_{\ell/X}\cong \O_{\P^1}(a_r-a_0) \oplus \cdots \oplus  \O_{\P^1}(a_r-a_{r-1})$. 

On voit en particulier que la courbe rationnelle $\ell$ est libre: ses d\'eformations recouvrent la vari\'et\'e $X$. On a donc
\begin{equation}\label{equ:section}
H^0(\ell, (T_{X/\P^1}^{\otimes i} \otimes \pi^\ast T^{\otimes p-i}\otimes  \O_{\P(E)}(-p-q)\otimes \pi^\ast N^{-1})|_{\ell} ) \neq 0. 
\end{equation}
Comme $\ell\cong \P^1$ est une section de $\pi : X \rightarrow \P^1$, on a 
\[
(T_{X/\P^1})|_{\ell}\cong N_{\ell/X}\cong \O_{\P^1}(a_r-a_0) \oplus \cdots \oplus  \O_{\P^1}(a_r-a_{r-1}).
\]
Par ailleurs on a $\pi^\ast T|_{\ell}\cong \O_{\P^1}(\a_0) \oplus \cdots \oplus \O_{\P^1}(\a_m)$, $\O_{\P(E)}(1)|_{\ell} \cong \O_{\P^1}(a_r)$, et $\pi^\ast N|_{\ell}\cong \O_{\P^1}(n)$.
La non-annulation (\ref{equ:section}) entra\^ine donc que
\[
(a_r-a_0)i + \a_0(p-i) + a_r(-p-q) -n \geq 0,
\]
et comme $\a_0 \leq 2$ on a 
\begin{equation}\label{equ:ineg1}
(a_r-a_0)i + 2(p-i) + a_r(-p-q) -n \geq 0, 
\end{equation}
d'o\`u 
\[
(a_r-a_0)i + 2(p-i) - a_r p \geq 0.
\]
On en d\'eduit d'une part que $i<p$, car
\[
2(p-i) \geq   a_r p -(a_r-a_0)i=  a_r(p-i)+ a_0 i> 0,
\]
et d'autre part que $ (a_r-a_0)i -(a_r -1)p \geq 0$, car $2i \geq p$. 

Comme $a_0 \geq 1$ on a donc
\[
(a_r- 1) (i-p) \geq (a_r-a_0)i -(a_r -1)p \geq 0,
\]
d'o\`u $a_r- 1 \leq 0$ puisque $i-p <0$. On a donc 
\[
a_0= a_1=\cdots = a_r =1.
\]
En tenant compte de cette information, l'in\'egalit\'e (\ref{equ:ineg1}) donne alors 
\[
p=2i  \  \text{ et } \   q=n=0.
\]
Par ailleurs  comme $i \geq \frac{r}{r+1} (p+q)$, on doit avoir $r=1$. Finalement on a $N = \O_{\P^1}$ et  $E = \O_{\P^1}(1)\oplus \O_{\P^1}(1)$. 

\end{proof}

\section{Deux th\'eor\`emes d'annulation}

Dans cette section on montre deux th\'eor\`emes d'annulation: l'un concerne le fibr\'e tangent relatif des fibr\'es projectifs sur une courbe, et l'autre le fibr\'e tangent relatif des fibr\'es en quadriques sur une courbe. Par fibr\'e en quadriques on entend un morphisme plat et projectif entre deux vari\'et\'es quasi-projectives, et dont toutes les fibres sont isomorphes \`a des quadriques irr\'eductibles et r\'eduites (mais pas n\'ecessairement lisses).

\begin{theo}\label{theo:3}
Soit $C$ une courbe projective lisse et connexe, $E$ un fibr\'e vectoriel ample sur $C$ de rang $r+1\geq 2$. Soit $\pi: X=\P(E) \rightarrow C$ le fibr\'e projectif associ\'e \`a $E$  et $G$ un fibr\'e vectoriel nef sur $X$. Alors 
\[
 \forall p\geq 1 \ , \ H^0(X, T_{X/C} ^{\otimes p} \otimes \O_{\P(E)}(-p)\otimes G^\ast) = 0
\]
\end{theo}

\begin{proof}
D'apr\`es \cite{CF}, il existe une courbe projective lisse et connexe $\tilde{C}$, un morphisme fini $f: \tilde{C} \rightarrow C$, un fibr\'e en droites $\tilde{M}$ ample sur $\tilde{C}$ et un entier $s\geq 1$ tel que $f^\ast E$ soit un quotient de $\tilde{M}^{\oplus s}$. Quitte \`a faire le changement de base $\tilde{C} \rightarrow C$ on peut donc supposer qu'il existe un fibr\'e en droites $M$ ample sur $C$ tel que $E\otimes M^{-1}$ soit engendr\'e par ses sections globales.

 On en d\'eduit que le fibr\'e en droites $N:= \O_{\P(E)}(1) \otimes \pi^\ast M^{-1}$ est lui aussi engendr\'e par ses sections globales. Consid\'erons $Z \in |N|$ un \'el\'ement g\'en\'eral du syst\`eme lin\'eaire sans point base $|N|$. Il suffira pour obtenir le r\'esultat de montrer que
\[
H^0(Z, (T_{X/C}^{\otimes p} \otimes \O_{\P(E)}(-p)\otimes G^\ast)|_Z ) = 0.
\]

Proc\`edons par r\'ecurrence sur $r$. Dans le cas o\`u $r=1$, $Z$ est une section de $\pi: X=\P(E) \rightarrow C$, donc $(T_{X/C})|_Z \cong N_{Z/X}\cong N|_Z $. Si $\pi_Z$ d\'esigne la restriction de $\pi$ \`a $Z$ on a donc
\[
H^0(Z, (T_{X/C}^{\otimes p} \otimes \O_{\P(E)}(-p)\otimes G^\ast)|_Z ) = H^0(Z, (\pi_Z ^\ast M^{p} \otimes G|_Z)^\ast ) = 0,
\] 
car $\pi_Z ^\ast M^{p} \otimes G|_Z$ est ample.

Si $r\geq 2$, on a pour une fibre g\'en\'erale $f$ de $\pi$ :
\[
(Z\cap f,  \O_{\P(E)}(1)|_{Z\cap f}) \cong (\P^{r-1}, \O_{\P^{r-1}}(1)), 
\]
donc d'apr\`es \cite[Corollary 5.4]{Fuj} il existe un fibr\'e vectoriel $F$ de rang $r$ sur $C$ tel que 
\[
(Z,  \O_{\P(E)}(1)|_Z) \cong (\P(F), \O_{\P(F)}(1)) .
\]
De plus on a la suite exacte suivante
\[
0 \rightarrow T_{Z/C} \rightarrow (T_{X/C})|_Z \rightarrow N_{Z/X} \cong N|_Z \rightarrow 0.
\]
D'apr\`es le lemme \ref{lem:1} il suffira donc de montrer que pour tout $i \in \{0,\ldots,p\}$ on a
\[
H^0(Z, T_{Z/C}^{\otimes i}\otimes N|_Z^{\otimes p-i} \otimes \O_{\P(F)}(-p)\otimes G^\ast)|_Z ) =0.
\]
ou encore
\[
H^0(Z, T_{Z/C}^{\otimes i}\otimes \O_{\P(F)}(-i)\otimes (\pi_Z ^\ast M^{p-i} \otimes G|_Z)^\ast ) =0.
\]

Si $i\geq 1$, le r\'esultat est donn\'e par l'hypoth\`ese de r\'ecurrence, car $\pi_Z ^\ast M^{p-i} \otimes G|_Z$ est un fibr\'e vectoriel nef sur $Z$.

Si $i=0$, on a encore
\[
H^0(Z, (\pi_Z ^\ast M^{p} \otimes G|_Z)^\ast ) =0 ,
\]
car pour toute courbe irr\'eductible $C' \subset Z$ non contract\'ee par $\pi_Z$, la restriction de $\pi_Z ^\ast M^{p} \otimes G|_Z$  \`a $C'$ est ample.
\end{proof}

Introduisons quelques notations. Soit $X$ une vari\'et\'e normale et $\F$ un faisceau sans torsion sur $X$. Soient $X'\subset X$ le plus grand ouvert sur lequel $\F$ est localement libre, et $i:X'\hookrightarrow X$ l'inclusion. Pour tout entier $p\geq 1$ on notera $\F^{[\otimes p]}:=i_\ast[(\F|_{X'})^{\otimes p}]$. Pour un morphisme de vari\'et\'es $\pi :X \rightarrow Y$ on notera $T_{X/Y}:=( \Omega^1_{X/Y})^\ast$.

\begin{theo}\label{theo:4}
Soient $X$ une vari\'et\'e projective normale, $C$ une courbe projective lisse et connexe, et $\pi: X \rightarrow C$ un fibr\'e en quadriques de dimension $r\geq 1$. Soient $L$ un fibr\'e en droites ample sur $X$ tel que pour toute fibre $f\cong Q_r$ de $\pi$ on ait $L|_f\cong \O_{Q_r}(1)$ et $G$ un fibr\'e vectoriel nef sur $X$. On a alors
\[
\forall p\geq 1 \ , \ H^0(X, T_{X/C} ^{[\otimes p]} \otimes L^{-p}\otimes G^\ast) = 0
\]
\end{theo}

\begin{proof} La d\'emonstration est analogue \`a celle du th\'eor\`eme \ref{theo:3}.

On note $X'$ l'ouvert de $X$ o\`u le morphisme $\pi$ est lisse. On remarque pour commencer que comme les fibres de $\pi$ sont des quadriques irr\'eductibles et r\'eduites, le lieu singulier $X \setminus X'$ de $\pi$ est de codimension au moins $2$ dans $X$. De plus $(T_{X/C})|_{X'}= T_{X'/C}$ \'etant localement libre, on a  $(T_{X/C} ^{[\otimes p]})|_{X'}= T_{X'/C}^{\otimes p}$, donc il suffira de montrer que
\[
H^0(X', T_{X'/C}^{\otimes p} \otimes L|_{X'}^{-p}\otimes G|_{X'}^\ast) = 0.
\]

Comme $h^0(f,L|_f)=h^0(Q_r,\O_{Q_r}(1)) = r+2$ pour toute fibre $f$ de $\pi$, le faisceau $E:= \pi_\ast L$ est localement libre de rang $r+2$. De plus il existe un plongement de $X$ dans $\P(E)$ au dessus de $C$ : 
\[
\xymatrix{X \ar[rd]^\pi \ar[rr]^\iota & & \P(E)\ar[ld] \\
& C & }
\]
On identifie dans la suite $X$ avec son image $\iota(X)$. On a alors $\O_{\P(E)}(1) |_X \cong L$, et  pour une fibre $f'\cong \P^{r+1}$ de la projection $\P(E) \rightarrow C$ on a $X \cap f' \cong Q_r$.  

Soit $f: \tilde{C} \rightarrow C$ un morphisme fini, avec $\tilde{C}$ une courbe projective lisse et connexe. On consid\`ere le changement de base
\[
\xymatrix{\P(f^\ast E) \ar[d] \ar[r]^g & \P(E) \ar[d]\\
\tilde{C} \ar[r]^f & C}
\]
On note $\tilde{X}:= g^{-1}(X)$, et $\tilde{\pi}:\tilde{X}\rightarrow \tilde{C}$ la restriction de la projection $\P(f^\ast E) \rightarrow \tilde{C}$.  Alors $\tilde{\pi}$ est un fibr\'e en quadriques, et la vari\'et\'e $\tilde{X}$ est normale d'apr\`es \cite[Proposition II.8.23]{Har}: en effet, c'est une hypersurface dans la vari\'et\'e lisse $\P(f^\ast E)$, et elle est lisse en codimension $1$.  
 
Gr\^ace \`a \cite{CF}, on peut donc supposer, quitte \`a faire un changement de base, qu'il existe un fibr\'e en droites $M$ ample sur $C$ tel que $E\otimes M^{-1}$ soit engendr\'e par ses sections globales. Le fibr\'e en droites $N:= L \otimes \pi^\ast M^{-1}$ est alors engendr\'e par ses sections globales, et on  consid\`ere $Z \in |N|$ un \'el\'ement g\'en\'eral. On pose $Z':= Z \cap X'$. Il suffira pour d\'emontrer le r\'esultat de voir que
\[
H^0(Z', (T_{X'/C}^{\otimes p} \otimes L|_{X'}^{-p}\otimes G|_{X'}^\ast)|_{Z'} ) = 0.
\]

On proc\`ede par r\'ecurrence sur $r$. Dans le cas o\`u $r=1$, les fibres de $\pi$ sont des coniques irr\'eductibles et r\'eduites, donc sont lisses. On a donc $X=X'$, et $Z=Z'$. De plus $(T_{X/C})|_Z \cong N_{Z/X}\cong N|_Z $, donc si $\pi_Z$ d\'esigne la restriction de $\pi$ \`a $Z$ on a 
\[
H^0(Z, (T_{X/C}^{\otimes p} \otimes \O_{\P(E)}(-p)\otimes G^\ast)|_Z ) = H^0(Z, (\pi_Z ^\ast M^{p} \otimes G|_Z)^\ast ) = 0,
\] 
car $\pi_Z ^\ast M^{p} \otimes G|_Z$ est ample.

Si $r\geq 2$, on a la suite exacte suivante
\[
0 \rightarrow T_{Z'/C} \rightarrow (T_{X'/C})|_{Z'} \rightarrow N_{Z'/X'} \cong N|_{Z'} \rightarrow 0.
\]
D'apr\`es le lemme \ref{lem:1} il suffira donc de montrer que pour tout $i \in \{0,\ldots,p\}$ on a
\[
H^0(Z', T_{Z'/C}^{\otimes i}\otimes N|_{Z'}^{\otimes p-i} \otimes L^{-p}\otimes G^\ast)|_{Z'} ) =0.
\]
ou encore
\[
H^0(Z', T_{Z'/C}^{\otimes i}\otimes L^{-i}\otimes (\pi_{Z} ^\ast M^{p-i} \otimes G|_{Z})|_{Z'}^\ast ) =0.
\]

Or on a pour une fibre $F$ de $\pi$ :
\[
(Z\cap F,  L|_{Z\cap F}) \cong (Q_{r-1}, \O_{Q_{r-1}}(1)), 
\]
donc le morphisme $\pi_Z:Z \rightarrow C$ satisfait les hypoth\`eses du lemme. Le r\'esultat est donc donn\'e par l'hypoth\`ese de r\'ecurrence si $i\geq 1$, car $\pi_Z ^\ast M^{p-i} \otimes G|_Z$ est un fibr\'e vectoriel nef sur $Z$.

Si $i=0$, on a 
\[
H^0(Z', (\pi_{Z}^\ast M^{p} \otimes G|_Z)|_{Z'}^\ast ) = H^0(Z,(\pi_{Z}^\ast M^{p} \otimes G|_Z)^\ast )
\]
car $\codim(Z\setminus Z') \geq 2$. De plus pour toute courbe irr\'eductible $C' \subset Z$ non contract\'ee par $\pi_Z$, la restriction de $\pi_Z ^\ast M^{p} \otimes G|_Z$  \`a $C'$ est ample, donc le terme de droite de l'\'egalit\'e ci-dessus est nul.
\end{proof}

\section{D\'emonstration des r\'esultats}
Dans cette section on donne la d\'emonstration des th\'eor\`emes \ref{theo:1} et \ref{theo:2}. L'un des outils majeurs dans cette d\'emonstration est une g\'en\'eralisation du th\'eor\`eme de Miyaoka (\cite[Corollary 8.6]{Miy}), qui fait l'objet de la prochaine sous-section. 

\subsection{Une g\'en\'eralisation du th\'eor\`eme de Miyaoka}
Rappelons d'abord une d\'efinition. Soit $H$ un diviseur ample sur une vari\'et\'e projective normale $X$ de dimension $n\geq 1$. On appelle courbe g\'en\'erale au sens de Mehta-Ramanathan, ou encore courbe MR-g\'en\'erale (relativement \`a $H$), toute courbe de la forme $C= H_1 \cap \cdots \cap H_{n-1}$, o\`u les $H_i \in |m_i H|$ sont g\'en\'eraux, et les $m_i$ sont des entiers arbitrairement grands.
On remarque que si $X_0 \subset X$ est un ouvert dont le compl\'ementaire est de codimension au moins $2$, alors une courbe MR-g\'en\'erale est contenue dans $X_0$. 

\begin{prop}\label{prop:1}
Soient $X$ une vari\'et\'e projective lisse de dimension $n\geq 1$, $X_0\subset X$ un ouvert dont le compl\'ementaire est de codimension au moins $2$, $Y_0$ une vari\'et\'e lisse de dimension $m \geq 1$, et $f:X_0 \rightarrow Y_0$ un morphisme \'equidimensionnel propre et dominant. On suppose que pour toute courbe MR-g\'en\'erale $C \subset X_0$ (relativement \`a un diviseur ample donn\'e $H$ sur $X$), le fibr\'e vectoriel $f^\ast\Omega^1_{Y_0}|_C $ n'est pas nef. 
Alors toute compactification projective de $Y_0$ est unir\'egl\'ee.
\end{prop}

\begin{proof}
Soient $m_1, \ldots m_{n-1}$ des entiers $\geq 1$, et pour tout $i \in \{1, \ldots , n-1\}$, soit $H_i \in |m_i H|$. On pose $Z= H_1 \cap \cdots \cap H_{n-m}$, $Z_0= X_0 \cap Z$, et on note $i:Z_0 \hookrightarrow Z$ l'inclusion, et $g:Z_0 \rightarrow Y_0$ la restriction de $f$ \`a $Z_0$.

Pour $m_i$ assez grand, et $H_i \in |m_iH|$ assez g\'en\'eraux, la vari\'et\'e $Z$ est lisse, le morphisme $g:Z_0 \rightarrow Y_0$ est fini, et $C:=H_1 \cap \cdots \cap H_{n-1} \subset Z_0$  est une courbe MR-g\'en\'erale. Par hypoth\`ese $g^\ast \Omega^1_{Y_0}|_C$ n'est pas nef, il existe donc un sous-faisceau de $g^\ast T_{Y_0}|_C$ de degr\'e strictement positif. On note $\F:=i_\ast g^\ast T_{Y_0}$. Ce faisceau est r\'eflexif d'apr\`es \cite[Corollary 7.1]{Har2}, et sa restriction \`a $C$ contient un sous-faisceau de pente strictement positive. On en d\'eduit que le premier terme $\E:=HN_1(\F)$ de la filtration de Harder-Narasimhan de $\F$ est de pente strictement positive. On rappelle en effet que puisque $C$ est une courbe MR-g\'en\'erale, on a $HN_1(\F|_C) = HN_1(\F)|_C$ gr\^ace au th\'eor\`eme de Mehta-Ramanathan (\cite[Theorem 6.1]{MR}).
Ceci entra\^ine gr\^ace aux propositions \cite[Proposition 29]{KSCT} et \cite[Proposition 30]{KSCT} que le faisceau $\E$, ainsi que le faisceau $\E \otimes \E \otimes ( \F/\E)^\ast$, sont amples en restriction \`a $C$. 

Soient $K$ une cl\^oture galoisienne du corps de fonctions $K(Z_0)$ sur $K(Y_0)$, $h:T\rightarrow Z$ la normalisation de $Z$ dans $K$, et $T_0:= h^{-1}(Z_0)$. Notons que comme $h$ est fini, on a $\codim(T\setminus T_0) \geq 2$, et le diviseur $h^\ast H$ est ample. Soit $\F' := (h^\ast\F)^{\ast \ast}$. Ce faisceau r\'eflexif contient $(h^\ast\E)^{\ast \ast}$ qui est de pente strictement positive (relativement \`a $h^\ast H$), donc $HN_1(\F')$ est aussi de pente strictement positive. Quitte \`a remplacer $Z$ et $Z_0$ par $T$ et $T_0$, $g$ par $h|_{T_0}$, $\F$ et $H$ par $\F'$ et $h^\ast H$, on peut donc supposer que $K(Z_0) \supset K( Y_0)$ est une extension galoisienne de groupe de Galois $G$. 

Par unicit\'e de la filtration de Harder-Narasimhan, le faisceau $\E$ est alors invariant sous l'action de $G$, donc quitte \`a remplacer $Z_0$ par un autre ouvert dont le compl\'ementaire est de codimension au moins $2$ dans $Z$, on peut supposer qu'il existe un sous-faisceau $\G$ de $T_{Y_0}$ tel que $\E=i_\ast g^\ast\G$. 

Montrons maintenant que $\G$ est un feuilletage, c'est \`a dire qu'il est stable par crochet de Lie. Il suffit pour cela de montrer que le tenseur d'O'Neill associ\'e 
\[
\G\otimes \G \rightarrow T_{Y_0}/ \G
\] 
est identiquement nul. Or la restriction \`a $C$ du faisceau $\E \otimes \E \otimes ( \F/\E)^\ast$ est ample, donc la restriction \`a $g(C)$ du faisceau 
\[
\H om(\G\otimes \G , T_{Y_0}/ \G) \cong (\G\otimes \G)^\ast \otimes T_{Y_0}/ \G
\]  
est anti-ample. De plus les d\'eformations de $g(C)$ recouvrent la vari\'et\'e $Y_0$, par cons\'equent on a $\Hom(\G\otimes \G , T_{Y_0}/ \G) = 0$. On en conclut que $\G$ est un feuilletage.
 
On a finalement un feuilletage $\G\subset T_{Y_0}$ qui est ample en restriction \`a $g(C)$ : on peut \'etendre ce feuilletage \`a une compactification projective $Y$ de $Y_0$, et le th\'eor\`eme \cite[Theorem 1]{KSCT} entra\^ine alors que par un point g\'en\'eral de $g(C)$ il passe une courbe rationnelle. Ceci d\'emontre que $Y$ est unir\'egl\'ee. 

\end{proof}

\subsection{D\'emonstration du th\'eor\`eme \ref{theo:2}}

\begin{proof}[D\'emonstration du th\'eor\`eme \ref{theo:2}]
On proc\`ede par r\'ecurrence sur $n$. Si $n=1$, on voit facilement que $X \cong \P^1$ et $L\cong \O_{\P^1}(1)$. 

Si $n\geq 2$, $X$ est unir\'egl\'ee d'apr\`es \cite[Corollary 8.6]{Miy}. Soit $H\subset \RatCurves^n(X)$ une famille couvrante minimale de courbes rationnelles sur $X$. Soit $[g] \in H$ un \'el\'ement g\'en\'eral, on peut alors \'ecrire $g^\ast T_X \cong \O_{\P^1}(2) \oplus  \O_{\P^1}(1)^{\oplus d} \oplus  \O_{\P^1}^{\oplus d'}$ avec $d,d' \geq 0$. Par ailleurs  
$L$ \'etant ample on a $g^\ast L \cong \O_{\P^1}(l)$ avec $l\geq 1$. Comme $H^0(X, T_X ^{\otimes p} \otimes L^{-k}) \neq 0$ on a $2p-lk \geq 0 $, on en d\'eduit que $l=1$ car $k>p$ par hypoth\`ese. Par cons\'equent $g^\ast L \cong \O_{\P^1}(1)$, et $H$ est compl\`ete. 

La premi\`ere \'etape consiste \`a montrer que $X$ a un nombre de Picard $\rho(X)$ \'egal \`a $1$. Supposons par l'absurde que $\rho(X)\geq 2$. On consid\`ere $\pi_0:X_0 \rightarrow Y_0$ le quotient $H$-rc de $X$, o\`u $X_0$ est un ouvert dense de $X$. On a alors $\dim Y_0 \geq 1$ d'apr\`es \cite[IV.3.13.3]{Kol}. Quitte \`a remplacer $Y_0$ par un ouvert plus petit, on peut supposer que $Y_0$ et $\pi_0$ sont lisses.

Le corollaire \ref{cor:1} entra\^ine l'existence d'un entier $i\in\{0,\ldots,p\}$ tel que 
\begin{equation}\label{equ:non_annulation}
H^0(X_0, T_{X_0/Y_0}^{\otimes i} \otimes  (\pi_0^\ast T_{Y_0})^{\otimes p-i}\otimes L|_{X_0}^{-k}) \neq 0.
\end{equation}
En restriction \`a une fibre $f$ de $\pi_0$ on obtient donc:
\[
H^0(f, T_{f}^{\otimes i} \otimes L|_f^{-k}) \neq 0 
\]
On remarque d'abord que ceci implique que $i\geq 1$. De plus, comme $\dim Y_0 \geq 1$ on a $d:=\dim f < n$, donc l'hypoth\`ese de r\'ecurrence entra\^ine que $f\cong \P^d$ et $L|_f \cong \O_{\P^d}(1)$. Si $E=\pi_{0\ast}(L|_{X_0})$ on a $X_0 \cong \P(E)$ et $L|_{X_0}\cong \O_{\P(E)}(1)$. De plus d'apr\`es \cite[Theorem 2.6]{ADK} on peut supposer que $\codim(X\setminus X_0) \geq 2$ puisque la famille $H$ est compl\`ete. 

Soit $Y$ une compactification projective de $Y_0$. Montrons par l'absurde que $Y$ est unir\'egl\'ee. Si ce n'est pas le cas, alors la proposition \ref{prop:1} entra\^ine que pour une courbe MR-g\'en\'erale $C \subset X_0$, le fibr\'e vectoriel $\pi_0^\ast \Omega^1_{Y_0} |_C$ est nef. Soient $\phi: C \rightarrow Y_0$ la restriction de $\pi_0$ \`a $C$, et $E_C:=\phi^\ast E$. On consid\`ere le changement de base
\[
\xymatrix{\P(E_C) \ar[d]^{\pi_C} \ar[r]^\psi & X_0 \ar[d]^{\pi_0}\\
C \ar[r]^\phi & Y_0}
\]
Le fibr\'e vectoriel sur $\P(E_C)$ 
\[
G:= (\psi^\ast \pi_0^\ast \Omega^1_{Y_0})^{\otimes p-i} \otimes \psi^\ast (L|_{X_0})^{k-i} 
\] 
est nef (en fait il est m\^eme ample) car $\psi^\ast \pi_0^\ast \Omega^1_{Y_0}$ est nef et $\psi^\ast (L|_{X_0})$ est ample. Par ailleurs gr\^ace \`a la non-annulation (\ref{equ:non_annulation}) on a
 \[
H^0(\P(E_C), \psi^\ast (T_{X_0/Y_0}^{\otimes i} \otimes  (\pi_0^\ast T_{Y_0})^{\otimes p-i}\otimes L|_{X_0}^{-k}) )\neq 0.
\]
d'o\`u 
\[
H^0(\P(E_C), T_{\P(E_C)/C}^{\otimes i} \otimes \O_{\P(E_C)}(-i)  \otimes G^\ast) \neq 0.
\]
On obtient donc une contradiction avec le th\'eor\`eme \ref{theo:3}.

La vari\'et\'e $Y$ est donc unir\'egl\'ee. Soit $H'\subset \RatCurves^n(Y)$ une famille couvrante minimale de courbes rationnelles sur $Y$. Comme $\pi_0$ est une fibration en $\P^d$, on peut trouver sur $X$ une famille couvrante de courbes rationnelles $H''\subset \RatCurves^n(X)$ telle que pour une courbe g\'en\'erale $\ell \in H''$ la courbe $\overline{ \pi_0(\ell \cap X_0)}$ soit un \'el\'ement de la famille $H'$. De plus comme $\codim(X\setminus X_0) \geq 2$ on a $\ell \subset X_0$ pour $\ell \in H''$ g\'en\'erale, donc une courbe g\'en\'erale de la famille $H'$ est enti\`erement contenue dans $Y_0$.

Soit $[h] \in H'$ un \'el\'ement g\'en\'eral, on peut alors \'ecrire $h^\ast T_{Y_0} \cong \O_{\P^1}(2) \oplus  \O_{\P^1}(1)^{\oplus e} \oplus  \O_{\P^1}^{\oplus e'}$ avec $e,e' \geq 0$. 
Notons $Z:= \P(h^\ast E)$, on consid\`ere le changement de base suivant.
\[
\xymatrix{Z \ar[d]^\rho \ar[r] & X_0 \ar[d]^{\pi_0}\\
\P^1 \ar[r]^h & Y_0}
\]
Comme $[h] \in H'$ est un \'el\'ement g\'en\'eral, on d\'eduit de (\ref{equ:non_annulation}) que 
\[ 
H^0(Z , T_{Z/\P^1}^{\otimes i} \otimes \rho^\ast(h^\ast T_{Y_0})^{\otimes p-i}\otimes  \O_{\P(h^\ast E)}(-k) ) \neq 0.
\]
Enfin en appliquant le lemme \ref{lem:3} on obtient $k=p$, ce qui contredit nos hypoth\`eses.

On a donc $\rho(X)=1$. D'apr\`es \cite[Lemma 6.2]{ADK}, $T_X$ contient un sous-faisceau sans torsion $\E$ tel que $\mu(\E) \geq \frac{\mu(L^k)}{p}$. Pour un \'el\'ement g\'en\'eral $[g]\in H$ on a donc 
\[
\frac{\deg (g^\ast \E)}{\rg(\E)} \geq  \frac{\deg(L^k)}{p}= \frac{k}{p} > 1
\]
Comme $g^\ast \E$ est un sous-faisceau de $g^\ast T_X \cong \O_{\P^1}(2) \oplus  \O_{\P^1}(1)^{\oplus d} \oplus  \O_{\P^1}^{\oplus d'}$ on en d\'eduit que $g^\ast \E$ est ample. 
D'apr\`es \cite[Proposition 2.7]{ADK}, le quotient $H$-rc $\pi_0 : X_0 \rightarrow Y_0$ de $X$ est donc un fibr\'e en $\P^n$ pour un entier $n\geq 1$, d'o\`u $X\cong \P^n$  puisque $\rho(X)=1$. Enfin on en d\'eduit facilement que $L\cong\O_{\P^n}(1)$.
\end{proof}

\subsection{D\'emonstration du th\'eor\`eme \ref{theo:1}}
La premi\`ere \'etape consiste \`a montrer que sous les hypoth\`eses du th\'eor\`eme \ref{theo:1}, une famille couvrante minimale de courbes rationnelles sur $X$ est compl\`ete.

\begin{prop}\label{prop:2}
Soient $X$ une vari\'et\'e projective lisse de dimension $\geq 2$, $L$ un fibr\'e en droites ample sur $X$. On suppose qu'il existe un entier $p\geq 1$ tel que $H^0(X, T_X^{\otimes p} \otimes L^{-p}) \neq 0$.
Si $H\subset \RatCurves^n(X)$ est une famille couvrante minimale de courbes rationnelles sur $X$, alors $L$ est de degr\'e $1$ sur les courbes de $H$, et en particulier $H$ est compl\`ete.
\end{prop}

\begin{proof}
Soit $\ell \in H$ un \'el\'ement g\'en\'eral et $\tilde{\ell}\cong\P^1$ sa normalis\'ee. On a alors $T_X|_{\tilde{\ell}} \cong T_{\tilde{\ell}} \oplus \O_{\P^1}(1)^{\oplus d} \oplus  \O_{\P^1}^{\oplus d'}$ pour des entiers $d,d' \geq 0$, avec  $T_{\tilde{\ell}} \cong \O_{\P^1}(2)$. Comme $H^0(\tilde{\ell}, (T_X|_{\tilde{\ell}})^{\otimes p} \otimes (L|_{\tilde{\ell}})^{-p}) \neq 0$, on a un morphisme non nul de faisceaux :
\[
\alpha : (L|_{\tilde{\ell}})^p \rightarrow (T_X|_{\tilde{\ell}})^{\otimes p} .
\]
Comme $(T_X|_{\tilde{\ell}})^{\otimes p}\cong \O_{\P^1}(2p) \oplus \bigoplus_i \O_{\P^1}(a_i)$ avec des entiers $a_i< 2p$, le degr\'e de $(L|_{\tilde{\ell}})^p$ est inf\'erieur ou \'egal \`a $2p$. De plus $L$ est ample, deux cas peuvent donc se pr\'esenter: soit $L|_{\tilde{\ell}} \cong \O_{\P^1}(1)$ (auquel cas $H$ est compl\`ete),
soit $L|_{\tilde{\ell}} \cong \O_{\P^1}(2)$. 

On va montrer que le deuxi\`eme cas est exclu : supposons par l'absurde que $L|_{\tilde{\ell}} \cong \O_{\P^1}(2)$. On a alors pour $ x\in \tilde{\ell}$ g\'en\'eral $\Im \alpha_x \cong T_{\tilde{\ell},x}^{\otimes p}$, et on en d\'eduit en particulier que $T_{\tilde{\ell},x}$ ne d\'epend pas de la courbe g\'en\'erale $\tilde{\ell}$ passant par $x$. On consid\`ere $\pi_{\U}: {\U} \rightarrow H$ la famille universelle associ\'ee \`a $H$, et $\eta_{\U} : \U \rightarrow X$ le morphisme naturel. On a un morphisme de faisceaux $\eta_{\U}^\ast \Omega^1_X \rightarrow \Omega^1_{\U} $ induit par la diff\'erentielle de $\eta_{\U}$, qui compos\'e avec la surjection $\Omega^1_{\U} \rightarrow \Omega^1_{\U/H}$ donne un morphisme $\eta_{\U}^\ast \Omega^1_X \rightarrow \Omega^1_{{\U}/H}$. Comme $\Hom(\eta_{\U}^\ast\Omega^1_X, \Omega^1_{\U/H}) \cong \Hom(\Omega^1_X, \eta_{\U\ast}\Omega^1_{\U/H})$ on a donc un morphisme 
\[
\varphi : \Omega^1_X \rightarrow \eta_{\U\ast}\Omega^1_{\U/H}.
\]
Notons $\F$ l'image de ce morphisme. Soit $x \in X$ un point g\'en\'eral, on note $\U_x:=\eta_{\U}^{-1}(x)$. On a alors un isomorphisme 
\[
\eta_{\U\ast}\Omega^1_{\U/H} \otimes \kk(x) \cong H^0(\U_x,\Omega^1_{\U/H}|_{\U_x}).  
\]
Si  $\omega \in \Omega^1_{X} \otimes \kk(x)$ alors l'image de $\omega$ par $\varphi_x : \Omega^1_X\otimes \kk(x) \rightarrow H^0(\U_x,\Omega^1_{\U/H}|_{\U_x})$ s'identifie \`a l'application qui \`a $u\in \U_x$ associe $(\omega \circ d_u\eta_\U)|_{T_{\U/H}\otimes \kk(u)}$. Or $d_u\eta_\U(T_{\U/H}\otimes \kk(u))\cong T_{\tilde{\ell}}\otimes \kk(x)$ est engendr\'e par un vecteur $v_x$ qui ne d\'epend pas de la courbe g\'en\'erale ${\tilde{\ell}}$ passant par $x$.
On a donc 
\[
\varphi_x (\omega) = 0  \ \Leftrightarrow  \ \omega(v_x) = 0,
\]
ce qui entra\^ine en particulier que le faisceau $\F$ est de rang 1. Si $G$ est la saturation du faisceau $\F^\ast$ dans $T_X$, alors $G$ est de rang $1$, et il est r\'eflexif, c'est donc un faisceau inversible d'apr\`es \cite[Proposition 1.9]{Har2}. 

On a construit de cette mani\`ere un feuilletage en courbes $\Omega^1_X \rightarrow G^{-1}$ sur $X$ dont les feuilles sont les courbes de la famille $H$ et tel que $L^p \subset G^p \subset T_X^{\otimes p}$. De plus on a $G\cdot \tilde{\ell}= L \cdot \tilde{\ell} = 2$ : en effet on a d'une part $G \cdot \tilde{\ell} \leq 2$ car $G|_{\tilde{\ell}} \hookrightarrow  T_X|_{\tilde{\ell}} \cong \O_{\P^1}(2) \oplus \O_{\P^1}(1)^{\oplus d} \oplus  \O_{\P^1}^{\oplus d'}$, et d'autre part $G^p \cdot \tilde{\ell} \geq L^p \cdot \tilde{\ell} = 2p$.

Le but de ce qui suit est d'obtenir une contradiction en se ramenant au cas d'une surface lisse. On rappelle qu'on a le diagramme commutatif suivant (voir \cite[II.2]{Kol} ):
\begin{equation}\label{diag}
\xymatrix{
 \P^1 \times \Hom^n_{\birat}(\P^1,X) \ar@/^0.6cm/[rr]^-\mu \ar[r]_-{U} \ar[d]& \Univ^{rc}(X) \ar[r]_-{\eta} \ar[d]^{\pi} & X \\
 \Hom^n_{\birat}(\P^1,X) \ar[r]_-{u} & \RatCurves^n(X)
}
\end{equation}
Soit $T$ la normalis\'ee d'une courbe irr\'eductible contenue dans $u^{-1}(H)$ et non contract\'ee par $u$, et soit $C$ une compactification projective lisse de $T$. On peut choisir la courbe $C$ de sorte que $g(C) \geq 1$. Le morphisme $\mu$ induit une application rationnelle 
 \[
\xymatrix{\P^1 \times C \ar@{-->}[r]^{\varphi} & X \times C} .
\]
On note $S\subset X \times C$ l'adh\'erence de l'image de $\varphi$, $p : X \times C \rightarrow X$ et $q: X \times C \rightarrow C$ les deux projections. Par construction de $G$, si $\tilde{\ell}$ est la normalisation d'une courbe g\'en\'erale $\ell \in H$ l'application compos\'ee
\[
\I_{\tilde{\ell}}/\I_{\tilde{\ell}}^2 \rightarrow \Omega^1_X |_{\tilde{\ell}}  \rightarrow G^{-1}|_{\tilde{\ell}}
\] 
est identiquement nulle. L'application compos\'ee
\[
\I_S/\I_S^2 \rightarrow \Omega^1_{X \times C}|_S \cong p^\ast\Omega^1_X|_S\oplus q^\ast\Omega^1_C|_S \rightarrow p^\ast\Omega^1_X|_S \rightarrow p^\ast G^{-1}|_S
\] 
est donc nulle sur un ouvert dense de $S$, et comme le faisceau $p^\ast G^{-1}|_S$ est sans torsion, elle est identiquement nulle. Il existe donc une factorisation
\[
\xymatrix{ \Omega^1_{X \times C}|_S \ar[rr]\ar[rd] & & p^\ast G^{-1}|_S\\ 
 & \Omega^1_S \ar[ru]&  }
\]
On obtient un feuilletage en courbes $\Omega^1_S \rightarrow G_S^{-1}:=p^\ast G^{-1}|_S$ sur la surface $S$. On notera dans la suite $L_S:=p^\ast L|_S$. Comme $L$ est ample et $p|_S$ est g\'en\'eriquement fini, $L_S$ est nef et grand. On remarque de plus que l'on a $L_S^p \subset G_S^p$. Si $f$ est une fibre g\'en\'erale de $q|_S$, la restriction de $p$ \`a $f$ est birationnelle sur son image, que l'on note $C$, et on a 
\[
L_S \cdot f = (p^\ast L|_S)\cdot f = L \cdot C = 2.
\]
On a donc \'egalement $G_S \cdot f =2$.

Soit $n:S'\rightarrow S$ la normalisation de $S$. D'apr\`es \cite[Lemme 1.2]{Dru} il existe un unique feuilletage $\Omega^1_{S'} \rightarrow n^\ast G_S^{-1}$ sur $S'$ qui \'etend $n^\ast \Omega^1_S \rightarrow n^\ast G_S^{-1}$. 

Soit $d:\bar{S}\rightarrow S'$ une d\'esingularisation minimale de $S'$. D'apr\`es \cite[Proposition 1.2]{BW}, on a un isomorphisme naturel $d_\ast T_{\bar{S}} \tilde{\rightarrow} (\Omega^1_{S'})^\ast$, donc il existe un unique feuilletage $\Omega^1_{\bar{S}} \rightarrow d^\ast n^\ast G_S^{-1}$ sur $\bar{S}$ qui \'etend $d^\ast  \Omega^1_{S'} \rightarrow d^\ast n^\ast G_S^{-1}$. Notons $G_{\bar{S}}:=d^\ast n^\ast G_S$, $L_{\bar{S}}:=d^\ast n^\ast L_S$, $p_{\bar{S}}:=p \circ n \circ d : \bar{S} \rightarrow X$ et $q_{\bar{S}}:=q \circ n \circ d : \bar{S} \rightarrow C$. Le morphisme $n \circ d$ \'etant g\'en\'eriquement fini, le fibr\'e $L_{\bar{S}}$ est encore nef et grand.

En r\'esum\'e, on a obtenu un feuilletage $\Omega^1_{\bar{S}} \rightarrow G_{\bar{S}}^{-1}$ sur la surface lisse $\bar{S}$, et un fibr\'e en droites nef et grand $L_{\bar{S}}$ tels que $L_{\bar{S}}^p \subset G_{\bar{S}}^p$. On a donc 
\[
H^0(\bar{S}, T_{\bar{S}}^{\otimes p} \otimes L_{\bar{S}}^{-p}) \neq 0.
\]
De plus pour $f\cong \P^1$ une fibre g\'en\'erale de $q_{\bar{S}}$ on a $L_{\bar{S}}\cdot f =G_{\bar{S}}\cdot f =2$.

Supposons que $\bar{S}$ n'est pas minimale, et consid\`erons $\varepsilon : \bar{S} \rightarrow S_1$ la contraction d'une $(-1)$-courbe. La $(-1)$-courbe contract\'ee ne domine pas $C$ car $g(C)\geq 1$, elle est donc contenue dans une fibre de $q_{\bar{S}}$, et on a une factorisation :
\[
\xymatrix{ \bar{S} \ar[rr]^\varepsilon \ar[rd]_{q_{\bar{S}}} & & S_1 \ar[ld]^{q_1}\\ 
 & C &  }
\]
 
Il existe un fibr\'e en droites $L_1$ sur $S_1$ et un entier $k$ tel que $L_{\bar{S}}= \varepsilon^\ast L_1(-kE)$, o\`u l'on a not\'e $E$ le diviseur exceptionnel de $\varepsilon$. Si $f$ est une fibre g\'en\'erique de $q_1$, on a $L_1\cdot f = 2$. Comme $L_{\bar{S}}$ est nef, on a $k\geq 0$, et $L_1$ est nef. De plus comme $L_{\bar{S}}$ est grand on a $L_{\bar{S}}^2= L_1^2 -k^2 >0$, d'o\`u $L_1^2 >0$, donc $L_1$ est grand.

On a une injection $T_{\bar{S}} \hookrightarrow \varepsilon^\ast T_{S_1}$ donc 
\[
H^0(\bar{S},\varepsilon^\ast T_{S_1}^{\otimes p}\otimes \varepsilon^\ast L_1^{-p}\otimes \O_{S_1}(pk E) )\neq 0, 
\]
et comme $\varepsilon_\ast \O_{\bar{S}}(pk E) \cong \O_{S_1}$ on en d\'eduit que
\[
H^0(S_1, T_{S_1}^{\otimes p}\otimes  L_1^{-p}) \neq 0. 
\]

On peut donc supposer que $\bar{S}$ est minimale, et que 
\[
H^0(\bar{S}, T_{\bar{S}}^{\otimes p} \otimes L_{\bar{S}}^{-p}) \neq 0,
\]
ce qui contredit la proposition (\ref{prop:3}) et termine la d\'emonstration. 

\end{proof}

\begin{prop}\label{prop:3}
Soit $\pi: S \rightarrow C$ une surface r\'egl\'ee. On note $f$ une fibre de $\pi$. Si $L$ est un fibr\'e en droites nef et grand tel que $L\cdot f=2$, alors pour tout $p\geq 1$
\[
H^0(S,T_{S}^{\otimes p}\otimes L^{-p} ) = 0. 
\] 
\end{prop} 

\begin{proof} On suppose par l'absurde que $H^0(S,T_{S}^{\otimes p}\otimes L^{-p} ) \neq 0$.
D'apr\`es le corollaire \ref{cor:1} il existe alors un entier $i\in\{0,\ldots, p\}$ tel que
\begin{equation}\label{eq:4}
H^0(S,T_{S/C}^{\otimes i} \otimes (\pi^\ast T_C)^{p-i} \otimes L^{-p} ) \neq 0.
\end{equation}

On \'ecrit $S\cong \P(E)$, avec $E$ un fibr\'e vectoriel de rang $2$ normalis\'e (voir \cite[2.8.1]{Har}). On note $e = \deg( \det E ^\ast )$. Soit $C_0$ l'image d'une section de $\pi$ telle que $\O_{S}(C_0) \cong \O_{\P(E)}(1)$. On a $C_0^2=-e$, et $C_0\cdot f=1$. Comme $L\cdot f=2$, on a $L \cong \O_{\P(E)}(2) \otimes \pi^\ast A$ avec $A\in \Pic(C)$. De plus $L$ \'etant nef et grand on a $L^2 > 0$, donc $(2C_0 +\deg(A)f)^2=4(-e + \deg(A)) >0$. 

Par ailleurs d'apr\`es \cite[Lemma V.2.10]{Har} on a $T_{S/C}\cong \O_{\P(E)}(2) \otimes \pi^\ast (\det E ^\ast)$, donc on a d'apr\`es (\ref{eq:4})
\[
H^0(S,\O_{\P(E)}(2i-2p)) \otimes \pi^\ast(\det(E)^{-i} \otimes T_C^{p-i} \otimes A^{-i} )) \neq 0.
\]
On en d\'eduit que $2i-2p \geq 0$, d'o\`u $i=p$. On a alors
\[
\begin{array}{lcl}
H^0(S,T_{S/C}^{\otimes p}\otimes L^{-p}) &\cong &H^0(S,\pi^\ast(\det(E)^{-p} \otimes A^{-p} ))\\
 &\cong& H^0(C,(\det(E)\otimes A)^{-p} ) .
\end{array} 
\]
Or le faisceau $(\det(E)\otimes A)^{-p} $ est de degr\'e $-p( -e +\deg(A)) < 0$, donc n'a pas de section non nulle, et on a une contradiction.
\end{proof}

On est maintenant en mesure de d\'emontrer le th\'eor\`eme principal.

\begin{proof}[D\'emonstration du th\'eor\`eme \ref{theo:1}]
On proc\`ede par r\'ecurrence sur $n$. Si $n=1$, on voit facilement que $X \cong \P^1$ et que $L\cong \O_{\P^1}(1)$ ou $L\cong \O_{\P^1}(2)$. 

Si $n\geq 2$, $X$ est unir\'egl\'ee d'apr\`es \cite[Corollary 8.6]{Miy}. Soit $H\subset \RatCurves^n(X)$ une famille couvrante minimale de courbes rationnelles sur $X$. D'apr\`es la proposition \ref{prop:2}, $L$ est de degr\'e 1 sur les courbes de $H$, et $H$ est compl\`ete. 

Si $\rho(X)=1$, le r\'esultat est donn\'e par \cite[Theorem 6.3]{ADK}. On suppose donc $\rho(X)\geq 2$, le but de ce qui suit \'etant de d\'emontrer que $X \cong \P^1 \times \P^1$.

Consid\'erons  $\pi_0:X_0 \rightarrow Y_0$ le quotient $H$-rc de $X$, on a $\dim Y_0 \geq 1$ d'apr\`es \cite[IV.3.13.3]{Kol}. Quitte \`a remplacer $Y_0$ par un ouvert plus petit, on peut supposer que $Y_0$ et $\pi_0$ sont lisses. D'apr\`es le corollaire \ref{cor:1} il existe $i\in\{0,\ldots,p\}$ tel que 
\[
H^0(X_0, T_{X_0/Y_0}^{\otimes i} \otimes  (\pi_0^\ast T_{Y_0})^{\otimes p-i}\otimes L|_{X_0}^{-p}) \neq 0.
\]
En restriction \`a une fibre $f$ de $\pi_0$ on obtient donc:
\[
H^0(f , T_{f}^{\otimes i} \otimes L|_f^{-p}) \neq 0 
\]

$\bullet$ \textbf{Premier cas : $i<p$}

Dans ce cas on a $f\cong \P^d$ et $L|_f \cong \O_{\P^d}(1)$ d'apr\`es le th\'eor\`eme \ref{theo:2}. Si $E=\pi_{0\ast}(L|_{X_0})$ on a alors $X_0 \cong \P(E)$ et $L|_{X_0}\cong \O_{\P(E)}(1)$. De plus d'apr\`es \cite[Theorem 2.6]{ADK} on peut supposer que $\codim(X\setminus X_0) \geq 2$ puisque la famille $H$ est compl\`ete. 

De m\^eme que dans la d\'emonstration du th\'eor\`eme \ref{theo:2} on peut voir que si $Y$ est une compactification projective de $Y_0$, alors $Y$ est unir\'egl\'ee, et qu'il existe une famille couvrante minimale de courbes rationnelles $K$ sur $Y$ telle qu'une courbe g\'en\'erale de la famille $K$ soit enti\`erement contenue dans $Y_0$. Si $[g] \in K$ un \'el\'ement g\'en\'eral, si $Z:= \P(g^\ast E)$, et si $\rho : Z \rightarrow \P^1$ est la projection naturelle on a
\[ 
H^0(Z , T_{Z/\P^1}^{\otimes i} \otimes \rho^\ast(g^\ast T_{Y_0})^{\otimes p-i}\otimes  \O_{\P(g^\ast E)}(-p) ) \neq 0.
\]
Par cons\'equent le lemme \ref{lem:3} entra\^ine que $g^\ast E \cong \O_{\P^1}(1) \oplus \O_{\P^1}(1)$, et que $ Z \cong \P^1 \times \P^1$. La vari\'et\'e $X$ admet donc une famille couvrante de courbes rationnelles $H' \subset \RatCurves^n(X)$ dont un \'el\'ement g\'en\'eral correspond \`a une droite de $Z  \cong \P^1 \times \P^1$ non contract\'ee par $\rho$. De plus $L|_Z \cong \O_{\P^1\times \P^1}(1,1)$ est de degr\'e $1$ sur les courbes de la famille $H'$, donc cette famille est compl\`ete.

Soit $\phi : X' \rightarrow Y'$ le quotient (H,H')-rc de $X$ (voir \cite{ADK} pour la d\'efinition), o\`u $X'$ est un ouvert dense de $X$. On peut supposer que $Y'$ et $\phi$ sont lisses. 
Une nouvelle application du corollaire \ref{cor:1} entra\^ine l'existence d'un entier $j\in\{0,\ldots,p\}$ tel que 
\[
H^0(X', T_{X'/Y'}^{\otimes j} \otimes  (\phi^\ast T_{Y'})^{\otimes p-j}\otimes L|_{X'}^{-p}) \neq 0.
\]
Si $f'$ est une fibre g\'en\'erale de $\phi$ on a donc
\[
H^0(f', T_{f'}^{\otimes j} \otimes  L|_{f'}^{-p}) \neq 0,
\]
ce qui entra\^ine par hypoth\`ese de r\'ecurrence que soit $f' \cong \P^d$, soit $f' \cong Q_d$, pour un entier $d\geq 2$. 

Supposons qu'on soit dans le premier cas : on peut appliquer encore une fois le raisonnement ci-dessus pour montrer que  $X' \cong\P(E')$ pour un certain fibr\'e vectoriel $E'$, puis qu'une compactification projective de $Y'$ est unir\'egl\'ee, et en se restreignant \`a une courbe rationnelle g\'en\'erale de $Y'$ on voit que $E'$ est de rang $2$, ce qui est impossible puisque les fibres de $\phi$ contiennent des sous-vari\'et\'es isomorphes \`a $ \P^1 \times \P^1$. 

On a donc $f' \cong Q_d$,  et $j=p$ d'apr\`es le th\'eor\`eme \ref{theo:2}, d'o\`u
\[ 
H^0(X', T_{X'/Y'}^{\otimes p} \otimes  L|_{X'}^{-p}) \neq 0.
\]
D'apr\`es \cite[Lemma 2.2]{ADK} on peut \'etendre $\phi$ en un morphisme \'equidimensionnel propre et surjectif avec des fibres irr\'eductibles et r\'eduites $\phi'' : X'' \rightarrow Y''$, o\`u $X''$ est un ouvert de $X$ dont le compl\'ementaire est de codimension au moins 2.

Supposons que $\dim Y'' \geq 1$. On consid\`ere alors la normalisation $C$ d'une courbe compl\`ete passant par un point g\'en\'eral de $Y''$, et on pose $X_C := X'' \times_{Y''} C$. On note $\phi_C:X_C \rightarrow C$ la projection naturelle et $L_C$ le tir\'e en arri\`ere \`a $X_C$ du fibr\'e $L$. D'apr\`es \cite[Corollary 5.5]{Fuj}, $\phi_C$ est un fibr\'e en quadriques, or on a
\[
H^0(X_C, T_{X_C/C}^{[\otimes p]} \otimes  L_C^{-p}) \neq 0, 
\]  
ce qui contredit la conclusion du th\'eor\`eme \ref{theo:4}. $Y''$ est donc r\'eduit \`a un point, et on a $X=f' \cong Q_d$. Enfin on a $d=2$ puisque $\rho(X) \geq 2$, d'o\`u $X \cong  \P^1 \times \P^1$.

\bigskip 
 
$\bullet$ \textbf{Deuxi\`eme cas : $i=p$}
Montrons que ce cas est impossible. On a 
\[
H^0(X_0, T_{X_0/Y_0}^{\otimes p} \otimes  L|_{X_0}^{-p}) \neq 0
\]
et 
\[
H^0(f , T_{f}^{\otimes p} \otimes L|_f^{-p}) \neq 0,
\]
donc par hypoth\`ese de r\'ecurrence $f \cong \P^d$ ou $f \cong Q_d$ pour un entier $d\geq 1$. De plus d'apr\`es la proposition \ref{prop:2} le cas o\`u $f\cong \P^1$ et $L|_f \cong \O_{\P^1}(2)$ est exclu.

Dans le cas o\`u $f \cong \P^d$, $\pi_0:X_0 \rightarrow Y_0$ est un fibr\'e projectif, et on peut supposer d'apr\`es \cite[Theorem 2.6]{ADK} que $\codim(X\setminus X_0) \geq 2$. Notons $C$ la normalisation d'une courbe compl\`ete passant par un point g\'en\'eral de $Y_0$, et $X_C := X_0 \times_{Y_0} C$. Soient $\pi_C:X_C \rightarrow C$ la projection naturelle et $L_C$ le tir\'e en arri\`ere \`a $X_C$ du fibr\'e $L$. Il existe alors un fibr\'e vectoriel $E_C$ sur $C$ tel que $X_C\cong \P(E_C)$ et $L_C \cong \O_{\P(E_C)}(1)$, et on a
\[
H^0(X_C, T_{X_C/C}^{\otimes p} \otimes  L_C^{-p}) \neq 0, 
\]  
ce qui entre en contradiction avec le th\'eor\`eme \ref{theo:3}. 
  
Dans le cas o\`u $f \cong Q_d$, d'apr\`es \cite[Lemma 2.2]{ADK} on peut \'etendre $\pi_0$ en un morphisme \'equidimensionnel propre et surjectif avec des fibres irr\'eductibles et r\'eduites $\pi':X' \rightarrow Y'$, avec $\codim(X\setminus X') \geq 2$. 

Soit $C$ la normalisation d'une courbe compl\`ete passant par un point g\'en\'eral de $Y'$, et soit $X_C := X' \times_{Y'} C$. On note $\pi_C:X_C \rightarrow C$  la projection naturelle, et $L_C$ est le tir\'e en arri\`ere \`a $X_C$ du fibr\'e $L$. D'apr\`es \cite[Corollary 5.5]{Fuj}, $\pi_C$ est un fibr\'e en quadriques, or on a
\[
H^0(X_C, T_{X_C/C}^{[\otimes p]} \otimes  L_C^{-p}) \neq 0, 
\]  
ce qui contredit cette fois encore le th\'eor\`eme \ref{theo:4}. 
\end{proof}


\begin{thebibliography}{WWWW}

\bibitem[ADK08]{ADK} C. Araujo, S. Druel, S.J. Kov\'acs. \emph{Cohomological characterizations of projective spaces and hyperquadrics}, Invent. Math. 174 (2008), 233-253. 

\bibitem[Be00]{Bea} A. Beauville. \emph{Symplectic singularities}, Invent. Math. 139  (2000), 541-549.
 
\bibitem[BW74]{BW} D.M. Burns Jr., J.M. Wahl. \emph{Local contributions to global deformations of surfaces}, Invent. Math. 26 (1974), 67-88. 

\bibitem[CF90]{CF} F. Campana, H. Flenner. \emph{A characterization of ample vector bundles on a curve}, Math. Ann. 287 (1990), 571-575. 

\bibitem[Dr04]{Dru} S. Druel. \emph{Caract\'erisation de l'espace projectif}, Manuscripta Math. 115, $n^o 1$ (2004), 19-30.
 
\bibitem[Fu75]{Fuj} T. Fujita. \emph{On the structure of polarized varieties with $\Delta$-genera zero}, J. Fac. Sci. Univ. Tokyo, Sect. IA Math. 22 (1975), 103-115.
 
\bibitem[Ha77]{Har} R. Hartshorne. \emph{Algebraic Geometry}, Graduate Texts in Mathematics 52, Springer-Verlag, 1977.

\bibitem[Ha82]{Har2} R. Hartshorne. \emph{Stable reflexive sheaves}, Invent. Math. 66 (1982), 165-190. 

\bibitem[HL97]{Huy} D. Huybrechts, M. Lehn. \emph{The geometry of moduli spaces of sheaves}, Aspects of math. Vol. E 31, Friedr. Viehweg and Sohn, 1997.

\bibitem[KSCT07]{KSCT} S. Kebekus, L. Sol\'a Conde, M. Toma. \emph{Rationally connected foliations after Bogomolov and McQuillan}, J. Algebr. Geom. 16 $n^o 1$ (2007), 65-81. 
 
\bibitem[Ko96]{Kol} J. Koll\'ar. \emph{Rational Curves on Algebraic Varieties}, Ergebnisse der Mathematik und ihrer Grenzgebiete 3. Folge, Springer-Verlag, 1996.

\bibitem[MR82]{MR} V.B. Mehta, A. Ramanathan. \emph{Semistable sheaves on projective varieties and their restriction to curves}, Math. Ann. 258 (1982), 213-224. 

\bibitem[Mi87]{Miy} Y. Miyaoka. \emph{Deformations of a morphism along a foliation and applications}, in Algebraic Geometry Bowdoin 1985, Proc. Sympos. Pure Math., Vol.46 (p.245-268), Am. Math.Soc., 1987.

\end{thebibliography}
\end{document}